\documentclass{amsart}
\usepackage{amsmath}
\usepackage{amssymb}
\usepackage{amscd}
\usepackage{amsthm}
\usepackage{tikz}
\usepackage{xcolor}
\usepackage[pagebackref]{hyperref}

\newtheorem{lem}{Lemma}
\newtheorem{prop}[lem]{Proposition}
\newtheorem{thm}[lem]{Theorem}
\newtheorem{conj}{Conjecture}
\newtheorem{cor}[lem]{Corollary}
\newcommand{\Z}{{\mathbb Z}}

\newcommand{\M}{\operatorname{M}}

\newcommand{\ov}{\overline}

\newcommand{\rsp}{\raisebox{0em}[2.7ex][1.3ex]{\rule{0em}{2ex} }}
\newcommand{\lra}{\longrightarrow}

\newcommand{\coker}{{\operatorname{coker}}}

\newcommand{\Ind}{{\operatorname{Ind}}}
\newcommand{\id}{{\operatorname{id}}}
\newcommand{\im}{{\operatorname{im}}\ }

\newcommand{\Cl}{{\operatorname{Cl}}}

\newcommand{\Gal}{{\operatorname{Gal}}}
\newcommand{\Ver}{{\operatorname{Ver}}}

\newcommand{\SD}{{\operatorname{SD}}}
\newcommand{\rk}{{\operatorname{rk}}}
\newcommand{\la}{\langle}
\newcommand{\ra}{\rangle}
\newcommand{\eps}{\varepsilon}

\newcommand{\fa}{\mathfrak a}

\newcommand{\cO}{\mathcal O}

\title{Transfer kernels and the Schur multiplier}
\author{F. Lemmermeyer}
\begin{document}
\maketitle

\begin{abstract}
  In this article we prove that if $H$ is a normal subgroup of
  a finite group $G$ with $G/H$ cyclic, then the kernel of the transfer map
  $G/G' \longrightarrow H/H'$ equals $(G:H)$ times the order of the
  cokernel of the natural map $M(H) \longrightarrow M(G)$.
\end{abstract}

In this article we will prove that the capitulation kernel in
cyclic unramified extensions is bounded in terms of the Schur
multiplier of Galois groups. We include the relevant background
information, and the actual proof takes up little more than
a page. We begin with the history of the transfer map, present
an elementary introduction to the Schur multiplier, describe the
connection with the capitulation kernel in unramified abelian
extensions of number fields, and give a proof of the conjecture
in two special cases (of which the proof for groups with commutator
groups $G'$ of order $2$ is longer than the general proof), and finally
give the proof of the main result, which actually takes up less than a page.

\section{The Transfer Map}

One of Hilbert's conjectures on what we now call Hilbert class fields
concerned the capitulation of ideal classes. If $k$ is a number field
with Hilbert class field $k^1$, then the principal ideal theorem
states that every ideal in $k$ becomes principal in $k^1$. More
generally, let $K/k$ be an unramified abelian extension, let $K^1$ be
the Hilbert class field of $K$, let $G = \Gal(K^1/k)$ be its Galois
group over $k$, and let $H = \Gal(K^1/K)$.  Since
$K \subseteq k^1 \subseteq K^1$, we have $G/G' \simeq \Gal(k^1/k)$.
The map $\iota_{k \lra K}: \Cl(k) \lra \Cl(K)$ induced by sending an
ideal $\fa$ in $\cO_k$ to $\fa \cO_K$ is called the transfer of ideal
classes; its kernel $\kappa_{K/k} = \ker \iota_{k \lra K}$ consists of
classes of ideals that become principal in $K$. Artin's reciprocity
law induces a corresponding homomorphism $\Ver$ (also called transfer)
from $G/G'$ to $H$:

$$ \begin{CD}
  \Cl(k) @>{\sim}>> G/G' \\
  @V{\iota_{k \to L}}VV @VV{\Ver}V \\
  \Cl(K) @>{\sim}>> H
\end{CD} $$
This allowed Artin to reduce the principal ideal theorem to a
statement in group theory: If $G$ is a finite group, then the transfer
map $\Ver: G/G' \lra G'/G''$ is the trivial map. The name transfer
(Verlagerung) was introduced by Helmut Hasse in his report \cite{Hasse}.
The transfer map is defined as follows. Let $H$ be a subgroup of
finite index $n$ in a finite group $G$; write $G = Hg_1 \cup \cdots
\cup Hg_n$. For each $g \in G$ and each $i$ there is an $h_i \in H$
such that $g_i g = h_i g_{\sigma(i)}$ for a permutation $\sigma$ of
$\{1,\ldots,n\}$.  Then $t(g) = h_1h_2 \cdots h_n H' $ defines a
homomorphism $G \lra H/H'$ whose kernel contains $G'$. Thus $t$
induces a homomorphism
$$ \Ver: G/G' \lra H/H' $$
called the transfer map.

The transfer map was first used by Issai Schur in his article
\cite{Schur02} published in 1902. It was Schur's student Arnold
Scholz who pointed out this fact in his Jahrbuch review (JFM 60.0895.01)
of Iyanaga's article \cite{Iyanaga} on the principal ideal theorem,
which appeared in 1934 (contrary to the claims in \cite[p.~166--167]{FLR},
it was not Iyanaga who had credited Schur with inventing the transfer map). 
Here's what happened: Issai Schur wanted to publish the results of his
PhD student Roberto Frucht in his journal, the Mathematische Zeitschrift.
His co-editor Leon Lichtenstein objected to publishing a thesis in the
journal and mentioned that Schur's recent related results had appeared
in Crelle's journal. Thus Schur asked Hasse whether he would publish
Frucht's manuscript, and Hasse agreed to do that. When reading the
submission (see \cite{Frucht}), Hasse noted similarities between some
of Frucht's formulas and those appearing in Furtw\"angler's proof of
the principal ideal theorem. He asked Schur whether he could explain
these similarities, but Schur did not answer this question immediately.

When Artin visited Schur in September 1930 (Artin's parents-in-law were
living in Berlin), Artin worked on the proper definition of his $L$-series
at ramified primes, and he talked with Schur about these things. Schur
picked up Hasse's report \cite{Hasse} and discovered that Artin's
transfer map in Furtw\"angler's proof was the same map he had used in 1902.
In a letter (see \cite{SchurL}) to Hasse written on December 30, 1930, he
explained the connection in detail.

We do not know whether Scholz learned about Schur's priority from
Hasse or from his former teacher himself. Scholz used the transfer map
after giving a talk in Hamburg in 1928, where he must have
talked to Artin. Scholz later gave his student G. Hannink a topic on
the transfer for his dissertation
({\em Verlagerung und Nichteinfachheit von Gruppen}, Kiel 1935).

Only in the 1970s was it discovered that the first example of
the transfer map is Gauss's Lemma in the third proof of the
quadratic reciprocity law (see \cite{Leut} and \cite{Wat}).

\section{The Schur multiplier}

The results in this section are well known; see e.g. \cite{Karp}.
The Schur multiplier of a finite group $G$ is an important invariant
that plays a major role in the investigation of class field towers.
A group extension
$$ \begin{CD}
  0 @>>> A @>>> \Gamma @>>> G @>>> 0
\end{CD} $$
is called a central extension if $A \subseteq Z(\Gamma)$ is contained
in the center of $\Gamma$; in particular, $A$ is abelian. The group
extension is called a stem extension if $A \subseteq \Gamma'$ is
contained in the commutator group $\Gamma'$ of $\Gamma$. Schur has
shown that for central stem extensions
$$ \begin{CD}
  0 @>>> A @>>> D @>>> G @>>> 0,
\end{CD} $$
the order of $A$ is bounded, and that the groups $A$ for which the order
of $A$ is maximal are all isomorphic. In this case, $A = \M(G)$ is called
the Schur multiplier of $G$, and the group extension $D$ is called a
Schur cove of $G$. The Klein $4$-group $G = C_2 \times C_2$ has
Schur multiplier $\M(G) \simeq C_2$, and its Schur cover groups are
the dihedral group $D_8$ and the quaternion group $Q_8$ of order $8$.

In our proof below, we will use the following standard facts
on Schur multipliers (see \cite{Karp}).
\begin{itemize}
\item[(F0)] $\M(G) \cong H_2(G,\Z)$. This is a finite abelian group
  annihilated by $|G|$.
\item[(F1)] (Hopf's formula) If $G \cong F/R$ with $F$ free, then
  $\M(G) \cong (R \cap F')/[F,R]$.
\item[(F2)] If $G \cong \bigoplus_{i=1}^r C_{m_i}$ is abelian, then
  $$ \M(G) \cong G \wedge G = \bigoplus_{i<j} C_{\gcd(m_i,m_j)}. $$
  In particular, $\M(G) = 1$ if $G$ is cyclic.
\item[(F3)] If $1\to A\to\Gamma\to G\to 1$ is a central stem extension,
  then any set of lifts to $\Gamma$ of a generating set of $G$ generates
  $\Gamma$. In fact, if $Y$ is the subgroup generated by the lifts, then
  $\Gamma = YA$; as $A$ is central, $\Gamma' = Y' \subseteq Y$, hence
  $A \subseteq \Gamma' \subseteq Y$ and $\Gamma = Y$.
\end{itemize}

Next we prove a special case of the Ganea sequence (see \cite{Karp}):

\begin{lem}\label{key}
  Let $Z$ be a central subgroup of a finite group $G$. Then
  $$\coker\bigl(\M(G)\to \M(G/Z)\bigr)\;\cong\;Z\cap G'.$$
\end{lem}

\begin{proof}
  Write $G=F/R$ and let $S\supseteq R$ be the preimage of $Z$, so
  $G/Z=F/S$ and $S/R=Z$. By (F1), $\M(G/Z)=(S\cap F')/[S,F]$, and
  $\M(G)\to \M(G/Z)$ is induced by the identity of $F$ (note
  $[R,F]\subseteq[S,F]$). Its image is $(R\cap F')[S,F]/[S,F]$, hence
  $$ \coker = (S\cap F') \big/\bigl((R\cap F')[S,F]\bigr). $$
  Consider $\varphi\colon S\cap F'\to S/R[S,F]$. An element $rc$ with
  $r\in R$, $c\in[S,F]\subseteq F'$ lies in $F'$ iff $r\in F'$; so
  $\ker\varphi=(S\cap F')\cap R[S,F]=(R\cap F')[S,F]$, and
  $\coker\cong\im\varphi$. Since $Z$ is central,
  $[S,F]\subseteq R$, hence $R[S,F]=R$ and $S/R[S,F]=S/R=Z$.  
  Finally $\im\varphi=\{\,fR: f\in F'\cap S\,\}=Z\cap F'R/R=Z\cap G'$
  (if $f\in F'$ and $fR\in S/R$ then $f\in S$). $\square$
\end{proof}

\section{Theorems of Furtw\"angler and Kisilevsky}

Furtw\"angler \cite{Furtw} observed that if a number field $k$ has
$2$-class group of type $(2,2)$, then the $2$-class group of its
Hilbert $2$-class field $k^1$ is trivial or cyclic; the $2$-class
field tower of $k$ then terminates at $k^1$ in the first case and at
the second Hilbert $2$-class field $k^2$ in the second.  In addition,
the Galois group of $k^2/k$ is either the Klein $4$-group $(2,2)$, a
dihedral group, a semi-dihedral group or a quaternion
group. Kisilevsky \cite{Kisilevsky} studied this situation in detail
and showed that the Galois group $\Gal(k^2/k)$ is determined by the
capitulation in the three unramified quadratic subextensions $k_1$,
$k_2$ and $k_3$ of $k^1/k$. In the following, let $\kappa_j$ denote
the capitulation kernel in $k_j/k$ and
$C_j = N_{k_j/k} \Cl_2(k_j) \subset \Cl_2(k)$ the subgroup of
ideal classes corresponding to $k_j/k$ via Artin's reciprocity law.
The following table contains the results of Furtw\"angler and Kisilevsky;
in the last column we have added the order of the Schur multiplier of
$G$:
$$ \begin{array}{c|l|c}
  \rsp G = \Gal(k^2/k) & \text{capitulation} & |\M(G)| \\ \hline
  \rsp (2,2) & \kappa_1 = \kappa_2 = \kappa_3 = \Cl_2(k) & 2 \\
  \rsp Q_8   & \kappa_1 = C_1, \kappa_2 = C_2, \kappa_3 = C_3 & 1 \\
  \rsp D_m   & \kappa_1 = \Cl_2(k), \kappa_2 = C_3, \kappa_3 = C_2 & 2 \\
  \rsp \SD_m & \kappa_1 \cap C_1 = 1, \kappa_2 = C_3, \kappa_3 = C_2 & 1 \\
  \rsp Q_m, m \ge 2^4 & \kappa_1 = C_1, \kappa_2 = C_3, \kappa_3 = C_2 & 1 
\end{array} $$
Here $D_m$, $Q_m$ and $\SD_m$ denote the dihedral, the quaternion and
the semi-dihedral group of order $m = 2^n$.  In the case $G = \SD_m$
we have $\kappa_1 = C_2$, where $k_2$ is the quadratic subextension of
$k^1$ for which $\Gal(k^2/k_1)$ is a dihedral group.

One cannot fail to notice that if the Schur multiplier $\M(G)$ is trivial,
then $|\kappa_j| = 2$ for each quadratic unramified subextension $k_j$
of $k^1/k$. Computing capitulation kernels for other groups of
order $16$ with trivial Schur multiplier suggests the following
conjecture:

\begin{conj}
  If the Galois group $G = \Gal(k^2/k)$ of the second Hilbert class field
  of a number field $k$ has trivial Schur multiplier, then in each
  unramified quadratic subextension $k_j/k$ of $k^1/k$, exactly one
  nontrivial ideal class capitulates:  
  $$ \M(G) = 1 \quad \text{implies} \quad |\kappa_j| = 2. $$
\end{conj}

We can formulate this conjecture without any reference to number fields
by replacing the capitulation kernel with the kernel of the transfer map:
If $G$ is a finite $2$-group with trivial Schur multiplier, then
$|\ker \Ver_{G/G' \lra H/H'}| = 2$ for each subgroup $H$ of index $2$ in $G$.

Computing the kernels of the transfer maps for $2$-groups of small order
with nontrivial Schur multiplier suggests that much more is true:
\begin{conj}\label{Con2}
  For every finite $2$-group and each subgroup $H$ of index $2$ in $G$ we have
\begin{equation}\label{C1}
  |\ker \Ver_{G/G' \lra H/H'}|  \le 2 \cdot |\M(G)|.
\end{equation}
\end{conj}
This conjecture goes back to \cite[(7.20)]{LemT}; there I have proved
that if $k$ is a number field whose $p$-class field tower terminates
with $k^\infty$ and where $G = \Gal(k^\infty/k)$, then for every
cyclic unramified $p$-extension $K/k$ we have
$$ (K:k) \le |\kappa_{K/k}| \le (K:k) \cdot |\M(G)|. $$
Jack Schmidt and Derek Holt have verified this inequality for all
$2$-groups of order $\le 256$.

The natural generalization of our conjecture to arbitrary $p$-groups
is the following result, which is a special case of Theorem~\ref{thmp}.

\begin{thm}\label{TM}
  For each finite $p$-group $G$ and each subgroup $H$ with index $p$ we have
  \begin{equation}\label{C2}
    |\ker \Ver_{G/G' \lra H/H'}|  \le p \cdot |\M(G)|.
  \end{equation}
\end{thm}

The more general conjecture that 
$$ |\ker \Ver_{G/G' \lra H/H'}|  \le (G:H) \cdot |\M(G)| $$
for non-cyclic subgroups is not true: the groups $[128, 23]$ and
$[128, 31]$ have Schur multiplier of order $2$ and possess two
subgroups of index $4$ in which the kernel of the transfer map has
order $16$.

\section{Simple Cases}

Theorem \ref{TM} is easy to prove when the group $G$ is abelian:

\begin{prop}
  If $G$ is an abelian $2$-group, then $|\ker \Ver|\le 2\,|\M(G)|$.  
\end{prop}

\begin{proof}  
  Write $G = H \cup tH$. Since $G$ is abelian, $g^t=g$. For $g \in H$
  the transfer is $\Ver(g)=g\,g^t=g^2$, and for $g \notin H$ it is
  $\Ver(g)=g^2$.  So $\Ver$ is the squaring map and $\ker \Ver = G[2]$.
  Write $G \cong \bigoplus_{i=1}^rC_{2^{a_i}}$; then $|\ker \Ver| = 2^r$.
  By (F2), $\M(G) \cong \bigoplus_{i<j} C_{2^{\min(a_i,a_j)}}$,
  hence $|\M(G)| \ge 2^{r(r-1)/2}$. But
  $$ 2^r \le 2 \cdot 2^{r(r-1)/2} \iff \frac{(r-1)(r-2)}{2} \ge 0, $$
  which is true for all $r\ge1$. Equality occurs exactly for $r=1$, or
  $r=2$ with $\min(a_1,a_2)=1$.
\end{proof}

A similar proof for groups with $|G'| = 2$ requires considerably more work;
we include it here because it involves classifying all $2$-groups $G$ with 
$|G'| = 2$ and trivial Schur multiplier.

\begin{prop}
  If $G$ is a $2$-group with $G'$ of order $2$, then
  $|\ker \Ver|\le 2\,|\M(G)|$.  
\end{prop}

\begin{proof}
  Assume that $G' = \la  z \ra $ (where $G'$ is central and $G' \le H$,
  so $H'\le G'$),
  $\ov{G} = G/G' \cong \bigoplus C_{2^{a_i}}$ of rank $r$, and $b=\min a_i$.
  \begin{enumerate}
  \item \textbf{$\ker \Ver \subset \ov{G}[2]$.} For $g\in H$, $g^t=g[g,t]$
    with $[g,t]\in G'$. Modulo $G'$ the transfer therefore becomes
    squaring $\ov{G}\to H/G'\subset\ov{G}$, so $\ker \Ver \subset \ov{G}[2]$
    and $|\ker \Ver|\le2^r$.
  \item \textbf{$|\M(G)| \ge \big|\ov{G} \wedge \ov{G}\big|/2$.}
    Apply Lemma~\ref{key} to $Z = G'$; then
    $\coker(M(G) \lra M(G/G')) = G'$ has order $2$, hence
    $|M(G) \ge \frac12|M(G/G') = 2^{r(r-1)/2-1}$, and for $r=2$,
    $|\M(G)| \ge 2^{b-1}$.    
\item \textbf{$r\ge3$.} Then $2^r\le2^{r(r-1)/2}\le2|\M(G)|$, so the
  conjecture holds.
\item \textbf{$r=2$, $b\ge2$.} Then $|\M(G)|\ge2$ and
  $|\ker \Ver | \le 4 \le2 |\M(G)|$.
\item \textbf{$r=2$, $b=1$} ($\ov{G}\cong C_{2^a}\times C_2$). The
  bound $4$ suffices if $|\M(G)| \ge 2$. The remaining case is $\M(G)=1$,
  where $\ker \Ver$ must be smaller than $\ov{G}[2]$.
\end{enumerate}
The groups in the remaining case are, as we will show below, the
quaternion group if $G$ has order $8$ and the modular groups of order
$n \ge 16$. For these groups, it is well known that the result is
true.
\end{proof}

Call a finite $2$-group $G$ a \emph{$P$-group} if $|G'|=2$ and
$\M(G)=1$. Let
$$ M_{2^n}=\la  x,y\mid x^{2^{n-1}}=y^2=1,\
   y^{-1}xy=x^{1+2^{n-2}}\ra , \quad n\ge 4. $$

\begin{lem}\label{cor}
  Let $G$ be a $P$-group.
  \begin{enumerate}
  \item[(a)] $\M(G/G')\cong G'\cong C_2$.
  \item[(b)] If $Z\le Z(G)$ has order $2$ and $Z\ne G'$, then $G/Z$ is
    again a $P$-group.
  \end{enumerate}
\end{lem}

\begin{proof}
  (a) $G'$ is central; apply Lemma~\ref{key} with $Z=G'$: $\M(G)=1$
  gives $\M(G/G')\cong G'$.

  (b) $Z\cap G'=1$, so by Lemma~\ref{key} $\M(G)\to \M(G/Z)$ is onto;
  thus $\M(G/Z)=1$. Also $(G/Z)'=G'Z/Z\cong G'$ has order
  $2$. $\square$
\end{proof}

\begin{thm}\label{main}
  The $P$-groups are exactly $Q_8$ and the modular groups $M_{2^n}$, $n\ge4$.
\end{thm}

\paragraph{Step 1: $G/G'\cong C_{2^a}\times C_2$.}
By Lemma~\ref{cor}(a), $|\M(\ov{G})|=2$ for
$\ov{G} = G/G' = \bigoplus_iC_{2^{a_i}}$ ($a_1\ge\dots\ge a_r\ge1$). By (F2),
$\sum_{i<j}\min(a_i,a_j)=1$. This forces $r=2$, $a_2=1$ (for $r\le1$
the sum is $0$; for $r\ge3$ it is $\ge3$). Put $a=a_1$; then
$|G|=2^{a+2}$.

\paragraph{Step 2: generators and relations.}
Choose $x,y\in G$ whose images generate $\ov{G}$ with orders $2^a$ and
$2$. By (F3) (applied to the central stem extension
$1 \to G' \to G \to \ov{G} \to 1$),  $G = \la  x,y \ra $.
The subgroup $G'$ is the normal
closure of $z:=[x,y]$; being central of order $2$, $G'=\la 
z\ra $, $z\ne1$, $z$ central. Also $x^{2^a},y^2\in G'$:
$$x^{2^a}=z^\alpha,\qquad
y^2=z^\beta,\qquad\alpha,\beta\in\{0,1\},\qquad y^{-1}xy=x[x,y]=xz.$$
Note that $x$ has order $2^a$ and $\la  x\ra \cap G'=1$ if
$\alpha=0$ ($x^k\in G'$ forces $2^a\mid k$).

\paragraph{Step 3: $\alpha=0$ is impossible.} We prove by induction on $a\ge1$: \emph{no $P$-group as in Step~2 has $\alpha=0$.}
\begin{itemize}
\item $a=1$: then $x^2=1$ and $x\notin G'$, $|G|=8$, $|G'|=2$. $Q_8$
  has a unique involution, which lies in $G'$; so $G\cong D_8$. But
  $|\M(D_8)|=2$, so $D_8$ is no $P$-group.
\item $a\ge2$: $x^2$ is central ($[x^2,y]=[x,y]^x[x,y]=z^2=1$), hence
  $Z=\la  x^{2^{a-1}}\ra \le\la  x^2\ra $ is central of
  order $2$, and $Z\cap G'\subseteq\la  x\ra \cap G'=1$. By
  Lemma~\ref{cor}(b), $G/Z$ is a $P$-group. In $G/Z$ the images of
  $x,y$ generate, $\bar x^{2^{a-1}}=1$, $(G/Z)'=\la  z\ra $, and
  the abelianization of $G/Z$ is $G/G'Z\cong C_{2^{a-1}}\times C_2$
  with $x,y$ of orders $2^{a-1},2$. So $G/Z$ is a $P$-group of the
  type of Step~2 with parameter $a-1$ and $\alpha=0$, contradicting
  the induction hypothesis.
\end{itemize}

\paragraph{Step 4: $\alpha=1$.}
Now $z=x^{2^a}$, $x$ has order $2^{a+1}$, and $y^{-1}xy=xz=x^e$ with $e=1+2^a$. Put $n=a+2$; then $e=1+2^{n-2}$, $|G|=2^n$.
\begin{itemize}
\item $a=1$ ($n=3$): $x^4=1$, $y^{-1}xy=x^3=x^{-1}$. If $\beta=0$ this
  is $D_8$, excluded. If $\beta=1$ then $y^2=x^2$ and $G=Q_8$.
\item $a\ge2$ ($n\ge4$): if $\beta=0$ then $x,y$ already satisfy the
  defining relations of $M_{2^n}$. If $\beta=1$ ($y^2=x^{2^a}$), put
  $y'=yx^{2^{a-1}}$. Since $x^ky=yx^{ek}$,
$$y'^2=y^2\,x^{(e+1)2^{a-1}}=x^{2^a}\,x^{2^a+2^{2a-1}}=x^{2^{a+1}+2^{2a-1}}=1,$$
  as $x^{2^{a+1}}=1$ and $2a-1\ge a+1$ for $a\ge2$; and
  $y'^{-1}xy'=x^e$. So again $x,y'$ satisfy the relations of
  $M_{2^n}$.
\end{itemize}
In the last case $G$ is a quotient of $M_{2^n}$ (generated by $x,y'$
satisfying its relations); since $|M_{2^n}| = 2^n = |G|$ (a split
extension of $\la  x \ra  \cong C_{2^{n-1}}$ by $C_2$, as
$e^2 \equiv 1 \bmod 2^{n-1}$), $G \cong M_{2^n}$.

\section{Proof of the Main Theorem}

Let $G$ be a finite group and $H \trianglelefteq G$ a normal subgroup
with $C = G/H = \la \tau\ra $ cyclic of order $p \ge 2$.
For a $p$-group $G$ and a subgroup $H$ of index $p$, this is automatic
($H$ is normal, and a group of prime order is cyclic).

Fix $t \in G$ and set $\tau = tH$. Since $G/H$ is abelian we have
$G'\subseteq H$, and since $G/H=\la tH\ra $ we have
$G = \la H, t \ra = \bigcup_{j=0}^{p-1}t^jH$. Put $A = H/H'$ (written
additively unless stated otherwise).  Since $A$ is abelian, $H$ acts
trivially on $A$ by conjugation, so the conjugation action of $G$ on
$A$ factors through $C$; this makes $A$ a $\Z[C]$-module, with
$\tau\bar h = \overline{t^{-1}ht}$ (independent of the choice of $t$).
We write $A^\tau = \ker(\tau-1)$ for the group of $\tau$-fixed points
and $\mathrm N=1+\tau+\dots+\tau^{p-1}\in\Z[C]$ for the norm element.
Let $\Ver \colon G/G'\to A$ be the transfer and $\M(G) = H_2(G,\Z)$
the Schur multiplier of $G$; put $\kappa = \ker\Ver$ and
$Q = \coker\bigl(\M(H)\to\M(G)\bigr)$.

\begin{thm}\label{thmp}
  Let $G$ be a finite group and $H$ a normal subgroup such that $G/H$
  is cyclic of order $p$. Then there is an exact sequence
  $$ 0 \lra \kappa \lra G/G' \xrightarrow{\ \Ver\ } A^\tau \lra Q \lra 0$$
  with $Q \cong A^\tau/\Ver(G/G')$, and
  $$ |\ker \Ver| = p \cdot |\coker(\M(H)\to \M(G))|. $$
  In particular, for a $p$-group $G$ and $[G:H]=p$,
  $$ |\ker \Ver| = p\,|\coker(\M(H)\to \M(G))| \le p\,|\M(G)|,$$
  with equality if and only if $\M(H)\to \M(G)$ is the zero map.
\end{thm}

\begin{cor}\label{cor:main}
  In the situation of Theorem~\ref{thmp}:
  \begin{enumerate}
  \item[(a)] $p$ divides $|\ker\Ver|$; in particular $|\ker\Ver|\ge p$;
  \item[(b)] $|\ker\Ver| \le p\,|\M(G)|$;
  \item[(c)] $|\ker\Ver| = p$ if and only if $\M(H)\to\M(G)$ is onto;
    thus $|\ker\Ver| = p$ whenever $\M(G) = 1$.
  \end{enumerate}
\end{cor}

\begin{proof}
  All three statements follow from $|\ker\Ver|=p\,|Q|$, since $Q$ is a
  quotient of $\M(G)$ and $|Q|=1$ iff $\M(H)\to\M(G)$ is onto.
\end{proof}

\begin{prop}
  Let $K/k$ be a cyclic unramified extension of number fields of
  degree $p$ (unramified at all places, including the infinite ones),
  let $K^1$ be the Hilbert class field of $K$, $G=\Gal(K^1/k)$ and
  $H=\Gal(K^1/K)$. Then $|\kappa_{K/k}|=|\ker\Ver|$ (Section~1), and
  by Iwasawa's formula \cite{Iwasawa}
  $|\kappa_{K/k}|=p\,(E_k:NE_K)$. Hence
  $$ (E_k:NE_K)=|\coker(\M(H)\to\M(G))|. $$
\end{prop}

\subsection*{Standard facts used}
We use the following standard results on the homology groups $H_n(G,B)$
\begin{itemize}
\item[(H1)] A short exact sequence of $G$-modules gives a long exact
  homology sequence;
\item[(H2)] \emph{Shapiro's lemma}: let $R=\Z[G/H]$ with $G$ acting by
  left translation.  Regard $\Z$ as the trivial
    \emph{$H$-module}; then $\Ind_H^G\Z=\Z[G]\otimes_{\Z[H]}\Z\cong R$
    (induced and coinduced modules agree because $(G:H)<\infty$), and
  Shapiro's lemma \cite[III.(6.2)]{Brown} gives natural isomorphisms
  $H_n(G,R)\cong H_n(H,\Z)$ for all $n$.  Under these isomorphisms,
  \begin{itemize}
  \item the augmentation $\eps\colon \Z[G/H] \to \Z$ induces the
    map $i_*\colon H_n(H,\Z)\to H_n(G,\Z)$ induced by the inclusion
    $H\subseteq G$; \cite[III.6]{Brown}
    \item the map $\nu\colon\Z\to\Z[G/H]$, $1\mapsto\sum_{gH\in
      G/H}gH$, induces the \emph{transfer} $H_n(G,\Z)\to H_n(H,\Z)$;
      \cite[III.9]{Brown}
    \item the $G$-endomorphism of $\Z[C]=\Z[G/H]$ given by
      multiplication by $\tau$ induces on $H_n(H,\Z)$ the map $c_{t*}$
      induced by conjugation by $t$; \cite[III.8]{Brown}.
  \end{itemize}
\item[(H3)] $H_0(G,\Z)=\Z$; $H_1(H,\Z)=H/H'$; $H_2(G,\Z)=\M(G)$ the
  Schur multiplier; on $H_1$ the transfer is the group-theoretic transfer
  $\Ver$; on $H_0$ the transfer $\Z\to\Z$ is multiplication by the index
  $(G:H)$ (\cite[III.9]{Brown}).
\end{itemize}

\subsection*{Two exact sequences}

Let $R=\Z[C]$ with $C = G/H$, regarded as a $G$-module via $G \to C$,
and let $\eps\colon R\to\Z$, $\sum a_j\tau^j\mapsto\sum a_j$ be the
augmentation with kernel $I$.  Then $I=(\tau-1)R$: the inclusion
  ``$\supseteq$'' is clear, and if $\sum a_j=0$ then $\sum
  a_j\tau^j=\sum a_j(\tau^j-1)=(\tau-1)\sum_j
  a_j(1+\tau+\dots+\tau^{j-1})$.  Because $C$ is cyclic, the
annihilator of $\tau-1$ in $R$ is $\Z\mathrm N$: indeed $(\tau-1)\sum
a_j\tau^j=\sum(a_{j-1}-a_j)\tau^j$ (indices mod $p$) vanishes iff all
$a_j$ are equal. Hence $x \mapsto (\tau-1)x$ gives a surjection $R \to
I$ with kernel $\Z \mathrm N$, and $\nu(1) = \mathrm N$. We get
\begin{align*}
  & \text{(I)}\quad 0\to I\xrightarrow{\ j\ }R
     \xrightarrow{\ \varepsilon\ }\Z\to0,\\
  & \text{(II)}\quad 0\to\Z\xrightarrow{\ \nu\ }R
     \xrightarrow{\ \pi\ }I\to0,\qquad \pi(x)=(\tau-1)x,
\end{align*}
and $j\pi$ is multiplication by $\tau-1$ on $R$.
All maps are $G$-equivariant: $j$ is an inclusion of submodules; $\eps(\tau\lambda)=\eps(\lambda)$
(so $G$ acts trivially on $\Z$ in (I)); $\tau\mathrm N=\mathrm N$ (so $\nu$ is $G$-linear);
and $\pi(\tau\lambda)=\tau(\tau-1)\lambda=\tau\pi(\lambda)$ because $\Z[C]$ is commutative.

\subsection*{Sequence (I)}
The long exact sequence attached to (I) contains
$$ \begin{CD}
  H_2(G,R) @>{\eps_*}>> H_2(G,\Z) @>{\partial}>> H_1(G,I) @>{j_*}>> H_1(G,R).
   \end{CD} $$
By (H2) and (H3) this reads
$$ \M(H) \xrightarrow{\ i_*\ } \M(G) \xrightarrow{\ \partial\ }
   H_1(G,I) \xrightarrow{\ j_*\ } H/H'. $$
Exactness at $\M(G)$ and at $H_1(G,I)$ shows that $\partial$ induces
an isomorphism
\begin{equation}\label{eqI}
  Q = \coker \bigl(\M(H)\to \M(G)\bigr)\ \xrightarrow{\ \sim\ }\
      \ker   \bigl(j_*\colon H_1(G,I) \to H/H'\bigr).
\end{equation}

\subsection*{Sequence (II)}
The relevant part of the long exact homology sequence of (II) is
$$ \begin{CD}
  G/G' @>{\Ver}>> H/H' @>{\pi_*}>> H_1(G,I) @>{\partial'}>> H_0(G,\Z) @>{\nu_*}>> H_0(G,R),
  \end{CD} $$
where we used $H_1(G,\Z)=G/G'$, $H_1(G,R)=H_1(H,\Z)=H/H'$ and $\nu_*=\Ver$ on $H_1$.
Under $H_0(G,\Z)=\Z$ and $H_0(G,R)\cong H_0(H,\Z)=\Z$ (the latter induced by $\eps$),
the map $\nu_*$ is multiplication by $p$, since $\nu(1)=\mathrm N\mapsto p$. Thus
$$ \begin{CD}
  G/G' @>{\Ver}>> H/H' @>{\pi_*}>> H_1(G,I) @>{\partial'}>> \Z @>{\cdot p}>> \Z.
  \end{CD} $$
Exactness at the fourth term gives $\im \partial' = 0$ since
multiplication by $p$ is injective on $\Z$, so $\partial' = 0$,
$\pi_*$ is onto, and exactness at $H/H'$ gives
\begin{equation}\label{eqII}
  \ker\pi_* = \Ver(G/G'), \qquad H_1(G,I) \simeq (H/H') / \Ver(G/G').
\end{equation}

\subsection*{Combining the two sequences}

Let us collect what we have so far.  Since $j \pi:R \lra I \lra R$
is multiplication by $\tau-1$ on $R$, the induced map on
$H_1(G,R) \simeq A = H/H'$ is the endomorphism
$j_*\pi_*: a \mapsto a^{\tau-1}$. Setting
$A[\tau] = \ker (\tau-1)$, the exact sequences extracted from
the long exact homology sequences fit into the following exact and
commutative diagram:
$$ \begin{CD}
  @.        @.        @.              @.               0 @. \\
  @.        @.        @.              @.           @VVV   @. \\
  @.        @.        @.              @.              Q @. \\
  @.        @.        @.              @.            @VVV @. \\
  0 @>>> \kappa @>>> G/G' @>{\Ver}>> H/H' @>{\pi_*}>> H_1(G,I) @>>> 0 \\
  @.       @.         @VVV   @V{\tau-1}VV   @V{j_*}VV @. \\
  @.       @.          0  @>>>  H/H' @>{\id}>> H/H' @>>> 0
\end{CD} $$
The commutativity of the left square follows from
$\pi_*\Ver=0$ and $j_*\pi_*=\tau-1$.

Applying the snake lemma then gives us an exact sequence
$$ \begin{CD}
  0 @>>> \kappa  @>>> G/G' @>{\Ver}>> A[\tau] @>>> Q @>>> 0
   \end{CD} $$
(which is easily proved by a simple diagram chase in this case),
and in particular we have
\begin{equation}\label{eqQ}
  Q \simeq  A[\tau]/\Ver(G/G').
  \end{equation}

\begin{prop}\label{countp}
  Let $L$ be the preimage of $(\tau-1)A$ under the natural projection
  $H\to A$. Then $G' = L$, and we have $|G/G'| = p\,|A^\tau|$.
\end{prop}

\begin{proof}
  $L\subseteq G'$: For $h\in H$ we have, writing $A$ multiplicatively,
  $(\tau-1)\ov{h}=t^{-1}hth^{-1}H'=[t,h^{-1}]H'$. Since $A$ is generated by
  the classes $\ov h$ and $\tau-1$ is a homomorphism, $(\tau-1)A$ is
  generated by these elements; hence $L$ is generated by $H'$ and the
  commutators $[t,h^{-1}]$, $h\in H$. All of these lie in $G'$, so
  $L \subseteq G'$.

  $L$ is normal in $G$: The subgroup $(\tau-1)A$ is a
  $\Z[C]$-submodule of $A$, because $\tau$ commutes with $\tau-1$.
  The conjugation action of $G$ on $A$ factors through $C$, hence
  $(\tau-1)A$ is stable under all of $G$.  So for $g\in G$ and $x\in
  L$ we have $g^{-1}xg\in H$ (as $H\trianglelefteq G$) and its class
  in $A$ lies in $(\tau-1)A$, i.e.\ $g^{-1}xg \in L$.

  $G' \subseteq L$: The image of $H$ in $G/L$ is abelian since
  $H' \subseteq L$.  For $h\in H$ we have $[t,h^{-1}]\in L$,
  i.e.\ $t^{-1}ht=h$ in $G/L$, so $t$ commutes with the image of every
  element of $H$. Since $G=\la H,t\ra $, the group $G/L$ is generated by
  the images of $H$ and $t$, which commute pairwise; hence $G/L$ is
  abelian and $G'\subseteq L$.

  Counting: We have shown that $G' = L$, so $(H:G')=|A/(\tau-1)A|=|A^\tau|$,
  because kernel and cokernel of an endomorphism of a finite group
  have the same order. Since $G'\subseteq H$ we get
  $|G:G'|=|G:H|\,|H:G'|=p\,|A^\tau|$.
\end{proof}

\subsection*{Proof of Theorem~\ref{thmp}}

The exact sequence $0 \to \kappa \to G/G' \to A^\tau \to Q\to0$ was
established in \eqref{eqQ}. Hence $|G/G'| = |\kappa|\,|\Ver(G/G')|$
and $|Q| = |A^\tau|/|\Ver(G/G')|$. By Proposition~\ref{countp},
$$|\ker \Ver| = \frac{|G/G'|}{|\Ver(G/G')|} = \frac{p\,|A^\tau|}{|\Ver(G/G')|}
    = p \,|Q|=p\,|\coker(\M(H)\to \M(G))|.$$
The cokernel is a quotient of $\M(G)$, giving the inequality, with
equality iff the cokernel is all of $\M(G)$, i.e.\ iff $\M(H)\to \M(G)$
is zero. \qed

\section*{Remarks}

In November 2025 I looked at Folk's article \cite{Folk}.  He said that
there is extra capitulation in a cyclic unramified extension $K/k$ if
the order of $\kappa_{K/k}$ is strictly larger than $(K:k)$. He showed
that extra capitulation implies the existence of certain central
extensions, which in turn, so I hoped, would show that the Schur
multiplier of the Galois group of $K^1/k$ cannot be trivial.
Before I found the time to think this through, the AI revolution happened.

Now that the transfer conjecture is proved we might ask conversely how
these field extensions corresponding to the cokernel of 
$\M(H) \lra \M(G)$ may be constructed.

\section{Acknowledgement}
The proof was found with the help of the free version of Claude in
October 2026. Over several sessions, Claude suggested three proofs:
the first used spectral sequences, the second replaced everything
by explicit calcuelcations based on Hopf's formula; I chose to present
the third version using standard homology groups in my own exposition.

I would like to dedicate this article to the memory of my dear friend
Chip Snyder (1946--2026), who kept returning to these questions during
the final years of his life. In particular, together with Elliot
Benjamin (see \cite{BS}) he proved that the transfer conjecture holds
in the following case: If $G$ is a finite $2$-group and $H$ is a maximal
abelian subgroup of $G$, then 
$ \rk \ker \Ver(G/G' \lra H/H') \le 1 + \rk \M(G). $

\end{document}